\documentclass[11pt]{amsart}

\usepackage{amsmath, amsthm, amssymb}
\usepackage{graphicx}
\usepackage{hyperref}

\usepackage[margin=1.25in]{geometry}
\usepackage[parfill]{parskip}
\usepackage{setspace}
\theoremstyle{plain}
\newtheorem{theorem}{Theorem}
\newtheorem{lemma}{Lemma}[section]
\newtheorem{corollary}{Corollary}[theorem]

\theoremstyle{definition}
\newtheorem{problem}{Question}[section]
\newtheorem{defn}{Definition}
\newtheorem{conj}{Conjecture}

\numberwithin{equation}{section}

\newcommand{\sig}{\sigma}
\newcommand{\perm}{\mathbf{P}}
\newcommand{\id}{\mathbf{id}}
\newcommand{\rev}{\mathbf{rev}}
\newcommand{\Arch}{\mathcal{A}}
\newcommand{\Abf}{\mathbf{A}}
\newcommand{\Pspace}{\mathcal{P}}

\newcommand{\eps}{\epsilon}
\newcommand{\uni}[1]{\mathrm{Uniform}[#1]}
\newcommand{\Par}{\Pi}
\newcommand{\M}{\mathcal{M}}
\newcommand{\R}{\mathbb{R}}

\newcommand{\RSN}{\mathbf{RSN}^n}
\newcommand{\Sn}{\mathfrak{S}_n}
\newcommand{\Uni}{\mathrm{Uniform}[-1,1]}
\newcommand{\C}{\mathbf{C}}

\newcommand{\Was}[1]{W(#1)}
\newcommand{\pr}[1]{\mathbb{P}\left [ #1 \right]}
\newcommand{\E}[1]{\mathbb{E}\left [ #1 \right ]}

\newcommand{\ind}[1]{\mathbf{1}_{\{#1\}}}
\newcommand{\len}[1]{\mathcal{E}[#1]}

\title{Geometry of Permutation Limits}

\author{Mustazee Rahman}

\address{Department of Mathematics,
	MIT, Cambridge, MA, USA.}
\email{mustazee@gmail.com}

\author{B\'alint Vir\'ag}
\address{Department of Mathematics, University of Toronto, Toronto, ON, Canada and \newline
	MTA Alfred R\'{e}nyi Institute of Mathematics, Budapest, Hungary.}
\email{balint@math.toronto.edu}

\author{M\'at\'e Vizer}
\address{MTA Alfred R\'{e}nyi Institute of Mathematics, Budapest, Hungary.}
\email{vizermate@gmail.com}

\date{}

\begin{document}

\maketitle

\begin{abstract}
This paper initiates a limit theory of permutation valued processes,
building on the recent theory of permutons.
We apply this to study the asymptotic behaviour of random sorting networks.
We prove that the Archimedean path, the conjectured
limit of random sorting networks, is the unique path from the identity to the
reverse permuton having minimal energy in an appropriate metric.
Together with a recent large deviations result (Kotowski, 2016), it implies
the Archimedean limit for the model of relaxed random sorting networks.
\end{abstract}

\section{Introduction} \label{sec:intro}

The objective of this paper is two-fold. First, to develop a limit theory of permutation valued
stochastic processes that it is applicable to study their asymptotic properties.
Second, to apply this theory to study random sorting networks.

\subsection{Permutations limits} \label{sec:intro1}

Recently, the language of permutons has been developed to study asymptotic properties of permutations.
Examples within this theory include the study of finite forcibility, pattern avoidance, property testing, pseudorandomness,
the Mallows model, etc. (see \cite{Kral, permuton, KKRW, St} and references therein).
We begin with a discussion of permutons as it is pertinent to Theorem \ref{thm:permutonprocess}, our first result.

Let $\Sn$ be the symmetric group of $n$ elements. The \emph{empirical measure} of $\sig \in \Sn$ is
\begin{equation} \label{eqn:empirical}
\mu^{\sig} = \frac{1}{n} \sum_{i} \delta_{\big ( \frac{2i}{n} - 1, \frac{2\sig(i)}{n} - 1 \big )}.
\end{equation}
This is a probability measure of $[-1,1]^2$. One defines a sequence of permutations
$\{ \sig_n \}$ with size $|\sig_n| \to \infty$ to converge if $\mu^{\sig_n}$ converges
weakly to a Borel measure $\mu$.

A \emph{permuton} $\mu$ is a Borel probability measure on $[-1,1]^2$ with
uniform marginals. (In the literature, this definition is often w.r.t. the unit square $[0,1]^2$
but it is convenient for us to use $[-1,1]^2$.) It is proven in \cite{permuton} that limits of
permutations in the above sense are permutons and that every permuton may be realized
as a limit of permutations.

In this paper we study not a sequence of single permutations but rather
a sequence of sequences of permutations. For an integer $n$, let $[n] = \{1, \ldots, n\}$.
Suppose that
\begin{equation} \label{eqn:permutationprocess}
\sig^n = \big ( \sig^n_t; \, t \in [t_n] \big )
\end{equation}
is a $\Sn$-valued sequence. We say that $\sig^n$ is a \emph{permutation process} of $\Sn$,
or simply a permutation process when there is no ambiguity. We always set 
$\sig^n_0$ to be the identity permutation. Our goal is to find an appropriate representation
of the limits of sequences of permutation processes growing both in the size, $n$, of the permutations
and the length, $t_n$, of the sequence.

An example of a permutation process is a \emph{sorting network}.
A sorting network of $\Sn$ is a path of minimal length from the identity
$\id_n = 1,\ldots,n$ to the reverse permutation $\rev_n = n, n-1, \ldots, 1$ in the Cayley graph
of $\Sn$ generated by the adjacent transpositions $(i,i+1)$, for $1 \leq i \leq n-1$.
It is a permutation process of length $t_n = \binom{n}{2}$. The adjacent transpositions
are also called swaps. An example of a sorting network is the bubble sort algorithm
applied to $\rev_n$ and viewed in reverse time.

Another example of a permutation process is the interchange process on finite paths.
In this setting the permutation process is random. As the path lengths tend to infinity,
the interchange process converges in probability to a limit: stationary Brownian motion on an interval.
This is explained in Section \ref{sec:randperm}.

\subsection{Limits of permutation processes}
Given a permutation process $\sig^n$ of $\Sn$ as in (\ref{eqn:permutationprocess}),
the rescaled trajectory of particle $i \in [n]$ is the function
\begin{equation} \label{eqn:trajectory}
T^n_i (t/t_n ) = \frac{2\sig^n_t(i)}{n} - 1 \;\;\text{for}\; t \in [t_n].
\end{equation}
After linearly interpolating between the discrete times $t/t_n$, we may consider
$T^n_i$ as a continuous function from $[0,1]$ to $[-1,1]$. The \emph{trajectory process}
of $\sig^n$, denoted $X^n$, is the trajectory of a particle chosen uniformly at random: 
\begin{equation} \label{eqn:trajectoryprocess}
X^n = \frac{1}{n} \, \sum_{i=1}^n \, \delta_{T^n_i}.
\end{equation}
Observe that for every $t \in [t_n]$ the distribution of $X^n(t/t_n)$ is uniform over
the set $\{ \frac{2i}{n}-1; i \in [n]\}$ due to $\sig^n_t$ being a permutation.
Moreover, $\sig^n$ can be reconstructed from $X^n$ and $t_n$.

Let $\C$ denote the space of continuous functions from $[0,1]$ to $[-1,1]$ in
the topology of uniform convergence. The trajectory process is then a Borel
probability measure on $\C$. Given a sequence of permutation processes $\{\sig^n\}$,
its limit is defined to be the weak limit of its associated trajectory processes as Borel probability measures
on $\C$. In other words, $\{\sig^n\}$ converges if there is a stochastic process $X = (X(t), 0 \leq t \leq 1)$
with continuous sample paths such that for every uniformly continuous and bounded $F : \C \to \R$,
\begin{equation} \label{eqn:weaklimit}
\E{F(X^n )} = \frac{1}{n} \sum_{i=1}^n F(T^n_i) \;\; \overset{n \to \infty}{\longrightarrow}\;\; \E{F(X)}.
\end{equation}
Our first result characterizes the limits of permutation processes.

\begin{defn} \label{def:permprocess}
A \emph{permuton process} is a $[-1,1]$-valued stochastic process $X = (X(t), 0 \leq t \leq 1)$
with continuous sample paths and such that $X(t) \sim \Uni$ for every $t$.
\end{defn}

\begin{theorem} \label{thm:permutonprocess}
For each $n$, let $\sig^n = (\sig^n_t; \,t \in [t_n])$ be a permutation process of $\Sn$ with
$t_n \to \infty$ as $n \to \infty$. Suppose $\{\sig^n\}$ converges to a limit $X$ in the
sense of (\ref{eqn:weaklimit}). Then $X$ is a permuton process. Conversely, given
any permuton process $X$, there is a sequence of permutation processes that converges to $X$.
\end{theorem}

Theorem \ref{thm:permutonprocess} extends the limit theory of single permutations
to permutation processes. Indeed, if a sequence of permutation processes $\{\sig^n\}$
has a limit $X$ then for every $s \in [0,1]$ the limit of $\sig^n_{\lfloor s \, t_n\rfloor}$ is
the permuton with the distribution of $(X(0), X(s))$. Moreover, for any set of times $s_1, \ldots, s_k$,
the empirical measure of the $k$-tuples $\big ( \sig^n_{\lfloor s_j \,t_n \rfloor}(i); 1 \leq j \leq k \big)$
as $i$ ranges over $[n]$ converges weakly as measures rescaled onto $[-1,1]^k$
to the distribution of $(X(s_1), \ldots, X(s_k))$.

\subsection{Random sorting networks} \label{sec:intro2}

Recall from the previous section that a sorting network of $\Sn$ is a shortest
path from $\id_n$ to $\rev_n$ in the Cayley graph generated by adjacent transpositions.
The number of sorting networks of $\Sn$ was enumerated by Stanley \cite{Stanley},
and later a combinatorial bijection with staircase shaped Young tableaux was provided
by Edelman and Greene \cite{EG}.

The number of permutations in a sorting network of $\Sn$ is always $N := \binom{n}{2}$.
A \emph{random sorting network} of $\Sn$ is a sorting network of $\Sn$ chosen uniformly at random.
We denote this random permutation process as $\RSN = (\RSN_t; 0 \leq t \leq N)$ (thus,
$\RSN_t$ is the $t$-th permutation in $\RSN$).

The asymptotic behaviour of $\RSN$ was first studied by Angel et.~al.~\cite{RSN}.
It is shown that, as $n\to\infty$, the spacetime process of swaps of $\RSN$ converges
to the product of semicircle law and Lebesgue measure. It is also shown that, in the limit,
the particle trajectories are H\"older-1/2 continuous, and the support of the permutation
matrix lies within a certain octagon. Additional results about the asymptotic behaviour of
$\RSN$ have since been proved; see for example \cite{AGH} and the references therein.
However, the main conjecture of \cite{RSN}, the Archimedean path conjecture, remains open.
To state the conjecture and our results we first introduce the Archimedean measure.

The \emph{Archimedean measure} is the unique probability measure on the
plane with the property that all of its projections onto lines through the origin have the
$\Uni$ distribution. Its density, supported on the unit disk, is given by $\big (2\pi \sqrt{1 - x^2 - y^2} \big)^{-1}\,dxdy$.
It is in fact the projection of the normalized surface area measure of the
2-sphere onto the unit disk. Let $(\mathbf{A}_x,\mathbf{A}_y)$ denote a random variable whose
distribution is the Archimedean measure. The \emph{Archimedean process}
$\Arch = (\Arch(t); 0 \leq t \leq 1)$ is the permuton process defined by
\begin{equation} \label{eqn:Archprocess}
\Arch(t) = \cos(\pi t)\,\mathbf{A}_x + \sin(\pi t)\,\mathbf{A}_y.
\end{equation}

The Archimedean path conjecture \cite[Conjecture 2]{RSN} states that for every $t$ the
random permutation $\RSN_{\lfloor tN \rfloor}$ converges to the deterministic
permuton $(\Arch(0), \Arch(t))$. The \emph{Archimedean path} is the
permuton valued path $\Abf = (\Abf(t); 0 \leq t \leq 1)$ such that
\begin{equation} \label{eqn:Apath}
\Abf(t) \sim (\Arch(0), \Arch(t))\;\; \text{for every}\;\; 0 \leq t \leq 1.
\end{equation}
Thus, the Archimedean path conjecture is that the empirical measures of permutations
in $\RSN$ converges to the Archimedean path; see Figure \ref{fig:archpath}.
Observe that the Archimedean process is a random sine curve.
The sine curve conjecture \cite[Conjecture 1]{RSN} asserts that the trajectories
of particles are close to random sine curves with high probability; see Figure \ref{fig:sinecurves}.

\begin{figure}[htpb]
	\begin{center}
		\includegraphics[scale = 0.45]{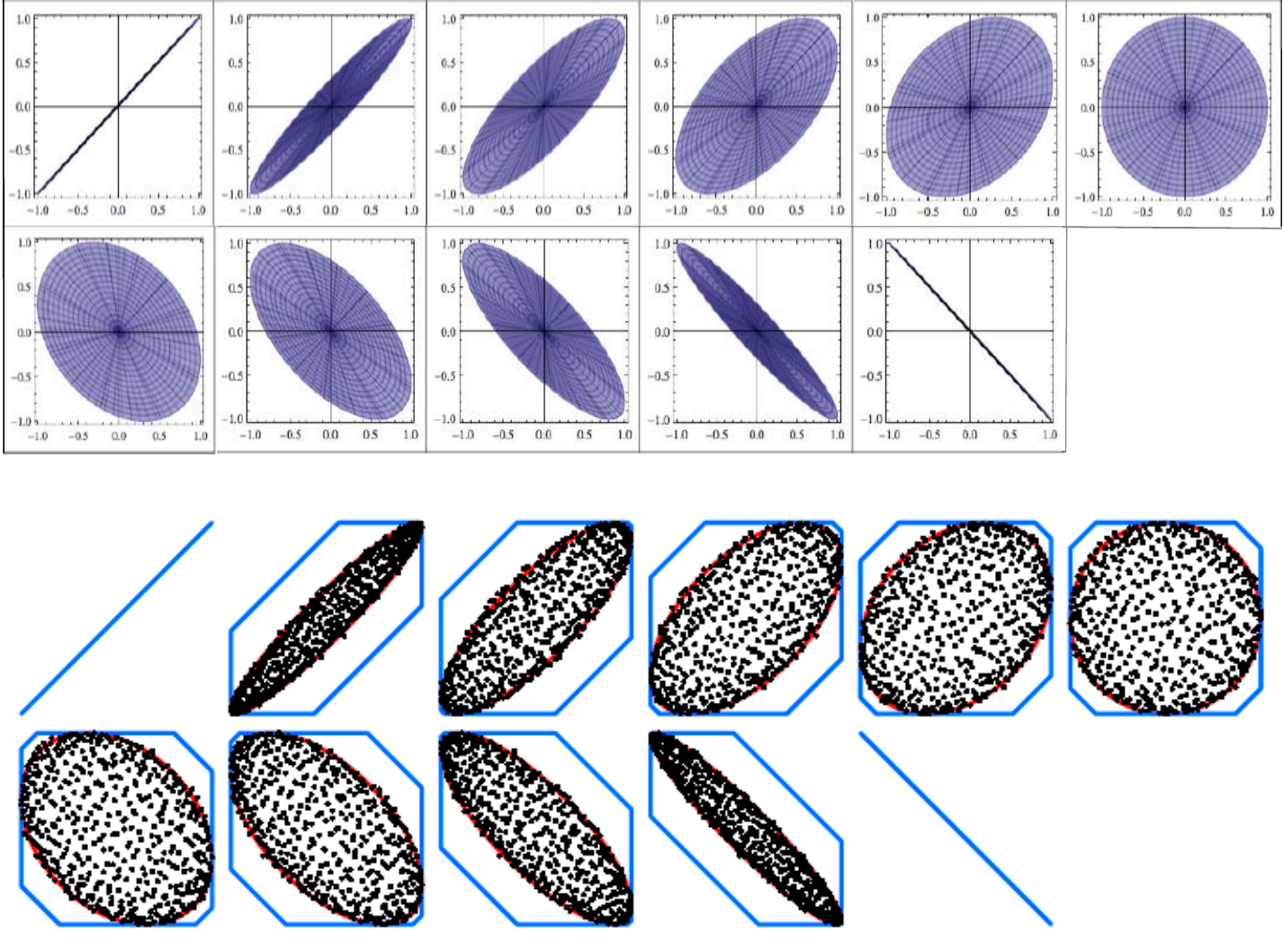}
		\caption{Support of the Archimedean path (top) and $\mathbf{RSN}^{500}$ (bottom).
		Bottom figure is from \cite[Figure 5]{RSN}.}
		\label{fig:archpath}
	\end{center}
\end{figure}

\begin{figure}[htpb]
	\begin{center}
		\includegraphics[scale = 0.45]{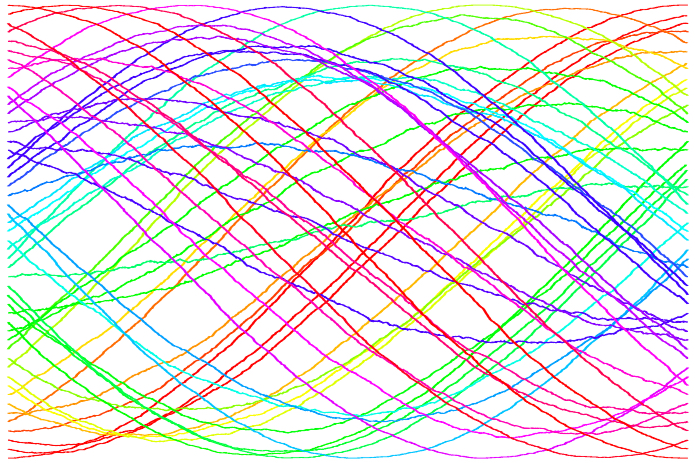}
		\caption{Some scaled particle trajectories from $\mathbf{RSN}^{2000}$ from \cite[Figure 1]{RSN}.}
		\label{fig:sinecurves}
	\end{center}
\end{figure}

The Archimedean path conjecture and the sine curve conjecture may be viewed
under a common framework as permuton processes as follows. It is a stochastic
process version of the Archimedean path conjecture (and implies it).
It also implies that typical trajectories of $\RSN$ are close to sine curves. 
\begin{conj} \label{conj}
$\RSN$ converges in probability as a permutation process to the Archimedean process (\ref{eqn:Archprocess}).
\end{conj}
We emphasize that this conjecture states the random trajectory process
of $\RSN$ concentrates around a deterministic limit, which is the Archimedean process.

\subsection{Variational characterization of Archimedean process}
The (Dirichlet) energy of a stochastic process $X = (X(t); 0 \leq t \leq 1)$ is
$$ \len{X} = \, \sup_{\Par} \; \sum_{i=1}^k \, \frac{\E{|X(t_i)-X(t_{i-1})|^2}}{t_{i} - t_{i-1}},$$
where the supremum is over all finite partitions $\Par = \{0 = t_0 < t_1 < \cdots < t_k =1 \}$ of $[0,1]$.
If $X$ has continuously differentiable sample paths then $\len{X} = \int_0^1 \E{X'(t)^2} \,dt$.
Thus, for example, a simple calculation shows that $\len{\Arch} = \pi^2/3$.

\begin{theorem} \label{thm:sinecurve}
Among all permuton processes $X$ with the property that $X(1) = -X(0)$,
the Archimedean process $\Arch$ uniquely minimizes the energy.
\end{theorem}

The theorem allows for a characterization of stationary random permutation processes which converge to
the Archimedean process in terms of the second moment of their speed. Random sorting networks are
invariant under $\eps$-shifts which take a trajectory $X(t)$ to $X(\eps+t)$. More precisely, the $\eps$-shift
makes the trajectory periodic with reversing boundary conditions, so the $\eps$-shift is defined as
$$\left((-1)^{\lfloor \eps+t \rfloor}X(\eps+t \text{ mod }1),\; 0\le t \le 1 \right).$$
We call a random trajectory process $\eps$-stationary if its distribution is invariant under
$\eps$-shift of all individual particle trajectories.

\begin{corollary}\label{speed}
Let $X^n$ be a tight sequence of random trajectory processes that are $\eps_n$-stationary with $\eps_n \to 0$.
Assume that
\begin{equation} \label{eqn:RSNconjecture}
\limsup_{t \to 0} \,\limsup_{n \to \infty} \;\frac{\E{(X^n(t)-X^n(0))^2}}{t^2} \leq \frac{\pi^2}{3}\,.
\end{equation}
Then the sequence converges in probability to the deterministic limit given by the Archimedean process.
\end{corollary}

The trajectory process of $\RSN$ is $N^{-1}$--stationary because of stationarity of the
swaps of $\RSN$, \cite[Theorem 1]{RSN}. Tightness of the random trajectory process of
$\RSN$ follows from \cite[Theorem 3]{RSN}, which states that for any $\delta > 0$,
with probability tending to 1, all individual trajectories $T^n$ of the particles in $\RSN$ satisfy
$$|T^n(t)-T^n(s)| \leq \sqrt{8}\,|s-t|^{1/2} + \delta \;\;\text{for every}\;s,t.$$
Corollary \ref{speed} is proved in Section \ref{sec:sinecurve}.

Random sorting networks can also be studied in the setting
of large deviation theory of the interchange process on paths. Consider
the discrete time interchange process on the $n$-path, which is the path graph
with $n$ vertices. A random sorting network is the interchange process on the $n$-path
conditioned to be at $\rev_n$ in the shortest possible time $N$. Instead, one can consider
\emph{relaxed random sorting network}, which is the interchange process conditioned to be
close to  $\rev_n$ in time $n^{2+\alpha}$ for some $\alpha \in (0,1)$. One can study the
relaxed network using large deviation theory in the following sense.

Suppose we fix a permuton process $X$ that satisfies $X(0) = - X(1)$.
One can ask what is the probability that the trajectory of a relaxed random
sorting network on the $n$-path is close to $X$. This is the problem addressed
in \cite{KV}. It is shown that this probability satisfies a large deviation principle
whose rate function is the energy of $X$. Then Theorem \ref{thm:sinecurve}
implies relaxed random sorting networks have to be close to the Archimedean process
with high probability. This proves the Archimedean path conjecture for the relaxed networks.

\subsection{Variational characterization of Archimedean path}
The 2-Wasserstein distance (henceforth, Wasserstein distance) between two
Borel probability measures $\mu, \nu$ on a metric space $K$ is defined by
\begin{equation} \label{eqn:Wasserstein}
\Was{\mu,\nu}^2 = \inf_{\text{couplings}\; (V,W) \;\text{s.th.}\; V \sim \mu, \, W \sim \nu} \E{d(V,W)^2}.
\end{equation}
We study permutons in the Wasserstein metric, whereby $K = [-1,1]^2$
in the Euclidean metric. Let $\id$ denote the identity permuton $(X,X)$ and $\rev$
denote the reverse permuton $(X,-X)$, where $X \sim \Uni$.

\begin{theorem} \label{thm:uniquegeodesic}
Let $\mu = (\mu(t); 0 \leq t \leq 1)$ be a permuton valued path from $\mu(0) = \id$
to $\mu(1) = \rev$. Then the energy of $\mu$ in the Wasserstein metric
satisfies $\len{\mu} \geq \len{\Abf} = \pi^2/6$, where $\Abf$ is the Archimedean path
(\ref{eqn:Apath}). If there is equality then $\mu(t) = \Abf(t)$ for every $t$.
\end{theorem}

The main tool used in proving Theorem \ref{thm:uniquegeodesic} should be of independent interest.
We show that for a permuton valued path $\mu$, there exists a $[-1,1]^2$--valued stochastic process $X$
such that the fixed time distributions of $X$ is given by $\mu$ and the energy of $X$ in the
$L^2$-metric equals the energy of $\mu$ in the Wasserstein metric. One may think of $X$
as being an optimal coupling of the measures along $\mu$. We prove such a `realization theorem'
for measure valued paths in a fairly general setting as stated in Theorem \ref{thm:realization}.

A motivation for Theorem \ref{thm:uniquegeodesic} is that the Wasserstein distance is a natural
metric on permutons. It is also related to sorting networks in the following way. A two-sided random sorting
network is a shortest sequence of permutations from $\id_n$ to $\rev_n$ so that in each step the permutation
is multiplied by an adjacent transposition either on left or on the right. This means that in each step,
two adjacent columns or two adjacent rows of the permutation matrix are exchanged. Thus the $1$s
in the permutation matrix can be thought of as particles moving horizontally or vertically. 

After scaling, we may consider the $[-1,1]^2$-valued trajectories for the $n$ particles in a uniformly
chosen two-sided sorting network. It can be shown that Conjecture \ref{conj} would imply that the
trajectory process of two-sided random sorting networks converges to an optimal coupling of the
Archimedean path $\Abf$.

\subsection{Permuton geometry} \label{sec:intro3}

Let $\Pspace$ denote the space of all permutons. Motivated by the great circle conjecture
about random sorting networks it is natural to study $\Pspace$ in the Wasserstein metric
since it is an infinite dimensional analogue of the permutohedron embedded into the Euclidean sphere.
(The permutohedron of order $n$ embeds naturally into an $(n-2)$-dimensional
sphere in $\R^n$.) In analogy with the sphere, one may ask whether the sum of distances squared,
$\Was{\id, \perm}^2 + \Was{\perm,\rev}^2$, is uniquely minimized by the Archimedean
measure over all $\perm \in \Pspace$?

\begin{theorem} \label{thm:distancesquared}
The function $\perm \mapsto \Was{\id,\perm}^2 + \Was{\perm,\rev}^2$ is minimized by a permuton $\perm \sim (X,Y)$
if and only if the pair $\big ( \frac{X-Y}{\sqrt{2}}, \frac{X+Y}{\sqrt{2}}\big)$ is also a permuton.
In particular, the Archimedean measure is not the unique minimizer.
\end{theorem}
In proving this theorem we will find a nice formula for the Wasserstein distance from
any permuton to the identity, as stated in Theorem \ref{lem:identitydistance}.

One can also ask whether there is a unique minimal energy path from $\id$
to a given permuton $\perm$? A motivation for this question is to understand
minimal length paths from the identity to arbitrary permutations of the permutohedron.
These are called reduced decompositions. Counting reduced compositions is a
deep and difficult combinatorial problem. We may get insights by studying related
questions in the space of permutons. For instance, are there analogues of the Archimedean path
conjecture for reduced decompositions of permutations approximating a
target permuton $\perm$? Can large deviation theory provide an asymptotic
count for the number of relaxed reduced decompositions of $\perm$,
\`{a} la sorting networks? We pose the following two open problems.

\begin{problem}[Uniqueness of minimal energy paths]
Under what condition does there exist an unique minimal energy path
in $\Pspace$ from $\id$ to a given permuton $\perm$? What about
for the Lebesgue permuton?
\end{problem}
\begin{problem}[Diameter of permuton space]
Suppose $\perm$ is a permuton. Does the minimal energy path(s) from $\id$ to $\perm$
have energy at least that of the Archimedean path?
\end{problem}

\subsection*{Remark}
Proof of the conjectures from \cite{RSN} have been announced in \cite{DD}. The proof uses
the framework of permuton processes and the main step involves proving Conjecture \ref{conj}.
The proof also relies on the local structure of random sorting networks and additional local-to-global
properties from \cite{ADHV, DV, GR}.

\subsection*{Outline of the paper} \label{sec:outline}

We prove Theorem \ref{thm:permutonprocess} in Section \ref{sec:permutonprocesses}.
In Section \ref{sec:energy} we define path energy in metric spaces and discuss some of its basic properties.
In Section \ref{sec:realization} we prove Theorem \ref{thm:realization} about realizing measure
valued paths as stochastic processes. In Section \ref{sec:sinecurve} we prove Theorem \ref{thm:sinecurve}
and Corollary \ref{speed}. In Section \ref{sec:ugproof} we prove Theorem \ref{thm:uniquegeodesic}.
Finally, in Section \ref{sec:permutongeometry} we prove Theorem \ref{thm:distancesquared}.

%%%%%%%%%%%%%%%%%%%%%%%%%%%%%%%%%%%%%%%%%%%%%%%%%%%%%%
%%%%%%%%%%%%%%%%%%%%%%%%%%%%%%%%%%%%%%%%%%%%%%%%%%%%%%
%%%%%%%%%%%%%%%%%%%%%%%%%%%%%%%%%%%%%%%%%%%%%%%%%%%%%%
\section{Limits of permutation processes} \label{sec:permutonprocesses}

In proving Theorem \ref{thm:permutonprocess} we state a lemma about
approximating continuous processes by their piecewise linear parts.
The proof is in the Appendix.

\begin{lemma} \label{lem:processfacts}
Let $Y = (Y(t); 0 \leq t \leq 1)$ be a continuous $[-1,1]$-valued process.
Consider its modulus of continuity $m^{\delta}(Y) = \sup_{s,t: |s-t| \leq \delta} |Y(s) - Y(t)|$.
Then $\E{m^{\delta}(Y)} \to 0$ as $\delta \to 0$. Moreover, if $Y$ and $\hat{Y}$ are continuous
processes then $|m^{\delta}(Y)-m^{\delta}(\hat{Y})| \leq 2||Y-\hat{Y}||_{\rm{\infty}}$.
Finally, let $\mathrm{Lin}(n,Y)$ be the process obtained from $Y$ such that it agrees with $Y$
at times $t = i/n$ for $0 \leq i \leq n$ and is linear in between. Then
$\E{||\mathrm{Lin}(n,Y) - Y||_{\rm{\infty}}} \to 0$ as $n \to \infty$.
\end{lemma}

\paragraph{\textbf{Proof that limit of permutation processes is a permuton processes.}}

Suppose that the trajectory processes $X^n$ of a sequence of permutation processes
$(\sig^n_t; t \in [t_n])$ converges to a continuous process $X$.
By Skorokhod's representation Theorem, we may assume
that the $X^n$ and $X$ are realized on a common probability space and
$||X^n-X||_{\rm{\infty}} \to 0$ almost surely. For a fixed $t$, we may choose
$s_n \in \{ i/t_n; 0 \leq i \leq t_n\}$ such that $|t-s_n| \leq 1/t_n$.
Then by triangle inequality and Lemma \ref{lem:processfacts},
$$|X(t) - X^n(s_n)| \leq ||X-X^n||_{\infty} + m^{1/t_n}(X^n) \leq 3 ||X-X^n||_{\infty} + m^{1/t_n}(X).$$
The term $m^{1/t_n}(X) \to 0$ almost surely in the sample outcomes of $X$ due to continuity,
and $||X^n-X||_{\infty} \to 0$ almost surely as well. Thus $|X(t) - X^n(s_n)| \to 0$ almost surely.
The distribution of $X^n(s_n)$ is uniform on the set $\{\frac{2i}{n} -1; i \in [n]\}$ as
remarked earlier. Therefore, $X^n(s_n)$ converges weakly to $\Uni$ and it follows that $X(t) \sim \Uni$.

\paragraph{\textbf{Proof that a permuton process is a limit of permutation processes.}}

Let $X = (X(t); 0 \leq t \leq 1)$ be the permuton process that is to be approximated
by permutation processes. We will construct a sequence of random permutation
processes and show that it converges almost surely to $X$. For $n \geq 1$ set
$\Par_n = \{ i/n; 0 \leq i \leq n\}$. The following defines a random permutation
process $(\sig^n_t; t \in [n])$ of $\Sn$.

Let $X_1, X_2, \ldots$ be i.i.d.~copies of $X$. We may assume that for every $t$ in the countable set
$\cup_n \Par_n$ the values $X_1(t), X_2(t), \ldots$ are all distinct. For each $n$,
consider the order statistics of $X_1(0), \ldots, X_n(0)$: $X_{(1)}(0) < X_{(2)}(0) < \cdots < X_{(n)}(0)$.
Let $\pi(i)$ be the index such that $X_{(i)}(0) = X_{\pi(i)}(0)$.  For $t \in [n]$, define the permutation $\sig^n_t$ by
$$\sig^n_t(i) = \;\text{rank of }\; X_{\pi(i)}(t/n) \;\text{among}\; X_1(t/n), \ldots, X_n(t/n).$$

For fixed $t \in [0,1]$, set $\Delta_{i,j} = \ind{X_j(t) \leq X_i(t)} - (X_i(t)+1)/2$.
As the rank of $x_i$ among $x_1, \ldots, x_n$ is $\sum_j \ind{x_j \leq x_i}$,
we have that for $t \in \Par_n$,
\begin{equation} \label{eqn:trajectorysum}
\frac{2\sig^n_{nt}(\pi^{-1}(i))}{n} -1 - X_{i}(t) = \frac{2}{n} \sum_{j=1}^n \Delta_{i,j}.
\end{equation}

Observe that $|\Delta_{i,j}| \leq 1$. Also, for $j \neq i$, $\E{\Delta_{i,j} \mid X_i(t)} = 0$.
This is where we use the fact that $X(t) \sim \Uni$ for every $t$. Moreover, for all $j$
such that $j \neq i$, the $\Delta_{i,j}$s are mutually independent conditional on $X_i(t)$.
Therefore by Bernstein's concentration inequality we infer that for $\eps \geq 0$ and every $i$,
\begin{equation} \label{eqn:Bernstein}
\pr{\frac{1}{n-1} \, | \sum_{j : j \neq i} \Delta_{i,j}| > \eps \; \Big |\; X_i(t)} \leq 2e^{-\frac{\eps^2(n-1)}{4}}.
\end{equation}

The right hand side (r.h.s.) of (\ref{eqn:trajectorysum}) is bounded in absolute value by
$\frac{2}{n} + \frac{2 | \sum_{j \neq i} \Delta_{i,j}|}{n-1}$. Therefore, taking an union bound
over all $t \in \Par_n$, setting $\eps = n^{-1/4}$ in (\ref{eqn:Bernstein}), then taking
expectation over $X_i(t)$ and another union bound over all particles $i = 1, \ldots, n$,
we infer that for all large $n$,
\begin{equation} \label{eqn:trajectorydeviation}
\pr {\sup_{1 \leq i \leq n, \, t \in \Par_n} \, \left | \frac{2\sig^n_{nt}(i)}{n}-1 - X_{\pi(i)}(t)\right | > 2n^{-1/4} + 2n^{-1}} \leq
2n^2\,e^{-\frac{n^{1/2}}{8}}.
\end{equation}

Recall that the trajectory of particle $i$ is $T^n_i(t) = (2/n)\sig^n_{nt}(i) - 1$ for $t \in \Par_n$,
and $T^n_i$ is linearly interpolated in between the times in $\Par_n$. Let
$$A_n = \sup_{1 \leq i \leq n, \; t \in \Par_n} \, | T^n_i(t) - X_{\pi(i)}(t)|\,.$$
Let $\mathrm{Lin}(n,X_i)$ be the piecewise linear function that agrees with $X_i$
at times $t \in \Par_n$. Observe that $||T^n_i-\mathrm{Lin}(n,X_{\pi(i)})||_{\infty} =
\sup_{t \in \Par_n} |T^n_i(t)-X_{\pi(i)}(t)|$ because both functions are
piecewise linear between the times $t \in \Par_n$. Let
$$B_n = \frac{1}{n}\, \sum_i ||X_i - \mathrm{Lin}(n,X_i)||_{\infty}.$$
From another application of Bernstein's inequality
%using that $||X_i - \mathrm{Lin}(n,X_i)||_{\infty} \leq 2$ for all $i$,
we deduce that
$$\pr{B_n > \E{||X-\mathrm{Lin}(n,X)||_{\infty}} + n^{-1/4}} \leq e^{-\frac{n^{1/2}}{16}}.$$

The r.h.s.~above is summable over $n$ and Lemma \ref{lem:processfacts}
implies that $\E{||X-\mathrm{Lin}(n,X)||_{\infty}} \to 0$ as $n \to \infty$. Furthermore,
$\pr{A_n \geq 4n^{-1/4}}$ is summable over $n$ due to (\ref{eqn:trajectorydeviation}).
Therefore the Borel-Cantelli lemma implies that there is a subset $\Omega$ of outcomes
of the $X_i$'s having probability 1 such that both $A_n, B_n \to 0$ if $\omega \in \Omega$.

Let $I \sim \mathrm{Uniform}([n])$ and let $\mathbb{E}_I$ denote expectation w.r.t.~$I$,
that is, the outcomes of the $X_k$s are kept fixed. Then,
\begin{align*}
\mathbb{E}_I [ || T^n_I - X_{\pi(I)}||_{\infty}] & \leq \mathbb{E}_I [|| T^n_I - \mathrm{Lin}(n,X_{\pi(I)})||_{\infty}]
+ \mathbb{E}_I [|| \mathrm{Lin}(n,X_{\pi(I)})-X_{\pi(I)}||_{\infty}]\\
& \leq A_n + B_n.
\end{align*}
Consequently, if $\omega \in \Omega$ then for every uniformly continuous and
bounded $F: \C \to \R$ we have that $\mathbb{E}_I [F(T^n_I) - F(X_{\pi(I)})] \to 0$.

The distribution of $X_{\pi(I)}$ over the random $I$ is the empirical measure
$(1/n) \sum_i \delta_{X_i}$ on $\C$. As $\C$ is a Polish space, the strong law of large
numbers for empirical measures on Polish spaces \cite{Kallenberg}
implies that there is a subset $\Omega'$ of outcomes of the $X_k$s having probability 1 such
that $X_{\pi(I)}$ converges weakly to $X$ if $\omega \in \Omega'$.

If $\omega \in \Omega \cap \Omega'$ then for every function $F$ as above we have that
$\mathbb{E}_I [F(T^n_I)] \to \E{F(X)}$. We thus conclude that almost surely in the outcomes of the $X_k$s,
the function $T^n_I$ converges weakly to $X$. This means that the sequence $\{\sig^n\}$
converges almost surely to the process $X$. Selecting any such good outcome of the $X_k$s
gives a deterministic sequence of permutation processes converging to $X$.

\subsection{Convergence of random permutation processes} \label{sec:randperm}

The limit notion for permutation processes naturally defines the limit notion
for random permutation processes. More precisely, if $\sig^n$ is a sequence
of random permutation processes of $\Sn$ then it converges if its trajectory processes
converge weakly as random measures on $\C$. The limit in this case is a
random permuton process, that is, a measure on permuton processes.
We illustrate this with two contrasting examples.

First, consider the interchange process, $\mathrm{Int}^n$, on the $n$-path.
It is a random permutation process of $\Sn$ generated by first sampling i.i.d.~uniform
random adjacent transpositions $\mathbf{\tau}_1, \mathbf{\tau}_2, \ldots$, and then setting
$\mathrm{Int}^n_t = \mathrm{Int}^n_{t-1} \circ \mathbf{\tau}_t$ for $t \geq 1$ with $\mathrm{Int}^n_0 = \id_n$.
The stationary distribution of this process is the uniform measure of $\Sn$ and its
relaxation time is of order $n^3$. Thus, one expects a limit of this process to exist if
it is run until time $n^3$.

It is shown in \cite{RV} that if $\mathrm{Int}^n$ is run until time $n^3$ then it converges to a
deterministic permuton process: stationary Brownian motion on $[-1,1]$.
(Actually, \cite{RV} considers the continuous time interchange process but the conclusion also holds
for the discrete time process.) Observe that the trajectory of each particle of $\mathrm{Int}^n$
is a simple random walk on the $n$-path. So by Donsker's Theorem \cite{Kallenberg},
after appropriate rescaling, each trajectory converges to Brownian motion on $[-1,1]$.
However, this alone does not imply that the trajectory process converges to stationary Brownian motion.

The convergence to the deterministic limit occurs because the trajectories become
``asymptotically independent", which means the following. Let $X^n_{\omega}$ be
the trajectory process for a sample outcome $\omega$ of $\mathrm{Int}^n$.
Let $T^n_{I_1}$ and $T^n_{I_2}$ be two samples from $X^n_{\omega}$, that is,
the trajectory of two particles $I_1$ and $I_2$ chosen independently and uniformly at random 
from $X^n_{\omega}$. Then, for every continuous and bounded function $F: \C \to \R$,
$$ \mathbb{E}_{\,\omega, I_1,I_2} \left [ F(T^n_{I_1}) F(T^n_{I_2}) \right] -
\mathbb{E}_{\,\omega, I_1} \left [ F(T^n_{I_1}) \right]^2 \;\; \overset{n \to \infty}{\longrightarrow} \; 0.$$
Asymptotic independence suffices to ensure that random permutation processes
have deterministic limits; see \cite{RV} for the details.

The second example illustrates a random limit of permutation process.
Consider $n$ particles placed on the vertices of the $n$-cycle.
At each time step, rotate the cycle one unit clockwise or counter clockwise
(by $2\pi /n$ radians) independently and uniformly at random.
This gives a random permutation process $\sig^n$ whereby each particle performs a simple random
walk on the $n$-cycle. However, note that the distances between particles remain fixed.
When run until time $n^3$ this process has the following limit.
Periodic Brownian motion on $[-1,1]$, denoted $B^{\mathrm{periodic}}$, is Brownian motion started
from a uniform random point of $[-1,1]$ and run in a period by identifying the endpoints $\pm 1$.
Let $U \sim \Uni$ be independent of $B^{\mathrm{periodic}}$. The limit of $\{\sig^n\}$ is the random
permuton process $\omega \to X_{\omega}$ such that for a sample outcome $U_{\omega}$ of $U$,
$$X_{\omega} \,\overset{law}{=} \, U_{\omega} + B^{\mathrm{periodic}} \;\; \big (\mathrm{mod}\; [-1,1] \, \big).$$
In other words, $[-1,1]$ is first rotated by $U_{\omega}$ (by identifying $\pm 1$) and then
rotated independently according a periodic Brownian motion. Two samples from
$X_{\omega}$ provide two periodic Brownian motions that start from a common point $\omega$--almost surely.

%%%%%%%%%%%%%%%%%%%%%%%%%%%%%%%%%%%%%%%%%%%%%%%%%%%%%
%%%%%%%%%%%%%%%%%%%%%%%%%%%%%%%%%%%%%%%%%%%%%%%%%%%%%
%%%%%%%%%%%%%%%%%%%%%%%%%%%%%%%%%%%%%%%%%%%%%%%%%%%%%
\section{Metric and energy for permutons} \label{sec:energy}

Let $(K,d)$ be a metric space. A path $\gamma = (\gamma(t); 0 \leq t \leq 1)$ is
a continuous function from the interval $[0,1]$ into $K$. A \emph{finite partition} of the interval
$[a,b]$ is a set of ordered points $\Par = \{ a = t_0 < t_1 < \cdots < t_n = b\}$.
Let $\rm{Part}[a,b]$ denote the set of all finite partitions of the interval $[a,b]$.
The \emph{mesh size} of a partition $\Par$ is $\Delta(\Par) = \max_{1 \leq i \leq n} \{ |t_i - t_{i-1}|\}$.

The energy of a path $\gamma$ with respect to a partition $\Par \in \rm{Part}[0,1]$ is
$$ \len{\gamma,\Par} = \sum_{i=1}^n \frac{d^2(\gamma(t_i), \gamma(t_{i-1}))}{t_i - t_{i-1}}\,.$$
The energy of $\gamma$, denoted $\len{\gamma}$, is
\begin{equation} \label{eqn:energydefn}
 \len{\gamma} = \sup_{\Par \in \rm{Part}[0,1]} \Big \{ \len{\gamma, \Par} \Big \} \,.
 \end{equation}
The energy of $\gamma$ restricted to the interval $[a,b]$ is
$$\len{\gamma,[a,b]} = \sup_{\Par \in \rm{Part}[a,b]} \Big \{ \len{\gamma, \Par} \Big \}.$$
Notice that if $a \leq b \leq c$ then we have
\begin{equation} \label{eqn:energyinequality}
\len{\gamma,[a,c]} \geq \len{\gamma,[a,b]} + \len{\gamma,[b,c]}.
\end{equation}
In particular, if $t_0 \leq t_1 \leq \ldots \leq t_n$ then $\len{\gamma,[t_0,t_n]} \geq \sum_{i=1}^n \len{\gamma,[t_{i-1},t_i]}$.

For partitions $\Par, \Par'$ of $[0,1]$, we write $\Par \subset \Par'$ ($\Par'$ is a refinement of $\Par$)
if $\Par'$ contains all points of $\Par$. The energy of a path is non-decreasing under refinements,
as Lemma \ref{lem:refinement} below shows. The proof is in the Appendix. We will use this lemma
throughout our arguments.

\begin{lemma} \label{lem:refinement}
Suppose $\Par \subset \Par'$ are two finite partitions of $[0,1]$. Then
$\len{\gamma,\Par} \leq \len{\gamma,\Par'}$ for any path $\gamma$ in $K$.
\end{lemma}

Using Lemma \ref{lem:refinement} we observe that for a path $\gamma$ there is a
sequence of nested finite partitions $\Par_0 = \{0,1\} \subset \Par_1 \subset \Par_2 \ldots$
such that $\len{\gamma,\Par_n} \nearrow \len{\gamma}$. We may also assume that
$\Delta(\Par_n) \to 0$. Thus $\cup_n \Par_n$ is a dense set of points in $[0,1]$.

We consider paths in two types of metric spaces. First, given a probability space $(\Omega, \Sigma, P)$ we
take $K = L^2(P, [-1,1])$, the Hilbert space of all square integrable random variables $Z : \Omega \to [-1,1]$.
A path in $K$ is then a stochastic process $X = (X(t); 0 \leq t \leq 1))$. Permuton processes
fall within this setup. Second, we take $K$ to be the space of permutons $\Pspace$
in the Wasserstein metric, which is the setting of Theorem \ref{thm:uniquegeodesic}.
The Wasserstein metric induces the topology of weak convergence on $\Pspace$
(see Lemma \ref{lem:Wconvergence}), and $\Pspace$ is compact in the weak topology
by Phokhorov's Theorem.

Finite energy paths in the space of Borel probability measures of a compact
metric space $K$, under Wasserstein metric, are related to finite energy $K$-valued
stochastic processes, in $L^2$ metric. We may realize the former as the latter
in an energy preserving manner. This is the content of Theorem \ref{thm:realization}
below, which is used to prove Theorem \ref{thm:uniquegeodesic}.

%%%%%%%%%%%%%%%%%%%%%%%%%%%%%%%%%%%%%%%%%%%%%%%%%%%%%%%
%%%%%%%%%%%%%%%%%%%%%%%%%%%%%%%%%%%%%%%%%%%%%%%%%%%%%%%%
%%%%%%%%%%%%%%%%%%%%%%%%%%%%%%%%%%%%%%%%%%%%%%%%%%%%%%%%
\subsection{Realizing measure valued paths as stochastic processes} \label{sec:realization}

Throughout this section $(K,d)$ denotes a compact metric space and $\M(K)$ denotes
the space of Borel probability measures on $K$ in the Wasserstein metric.
A path $\gamma = (\gamma(t); 0 \leq t \leq 1)$ in $\M(K)$ is
\emph{realized} by a $K$-valued stochastic process $X = (X(t); 0 \leq t \leq 1)$ if the following conditions holds.
\begin{enumerate}
\item $X(t) \sim \gamma(t)$ for every $t$.
\item $X$ has continuous sample paths almost surely.
\end{enumerate}
The energy of $X$ is as given by (\ref{eqn:energydefn}) with respect to the $L^2$ metric:
$$d_{L^2}(X(t),X(s)) := \E{d(X(t),X(s))^2}^{1/2}.$$

\begin{theorem} \label{thm:realization}
Suppose that $\gamma$ is a $\M(K)$--valued path with finite energy
with respect to the Wasserstein distance. There is a stochastic process
$X$ that realizes $\gamma$ in a energy preserving manner: $\len{X} = \len{\gamma}$.
\vskip 0.1in
We refer to the process $X$ as an \emph{optimal coupling} of $\gamma$.
\end{theorem}

The rest of the section proves Theorem \ref{thm:realization}.
To begin, note the following fact. If $\nu, \nu' \in \M(K)$
are two Borel probability measures then there is a coupling  $(V,W)$ of $\nu$ with $\nu'$
such that $\Was{\nu,\nu'} = \E{d(V,W)^2}^{1/2}$. This is because $K$ is compact.
Using this fact we inductively build up optimal couplings by using the following lemma.

\begin{lemma} \label{lem:finiterealization}
Let $\gamma_0, \ldots, \gamma_n \in \M(K)$. There exist jointly distributed $K$--valued
random variables $(X_0, \ldots, X_n)$ such that $X_i \sim \gamma_i$ and
$\E{d(X_{i-1},X_i)^2} = \Was{\gamma_{i-1},\gamma_i}^2$ for $1 \leq i \leq n$.
\end{lemma}

\begin{proof}
We proceed by induction. The case for two measures is mentioned above (see Lemma \ref{lem:bestcoupling}).
To carry out the induction step we will need the following measure theoretic fact. It is often known as the
Disintegration Theorem (see \cite[Theorem 5.10]{Kallenberg}).

\paragraph{\textbf{Fact:}} Let $(X,Y) \in K^2$ be jointly distributed random variables. There is a
measurable function $g : K \times [0,1] \to K^2$ such that if $U \sim \mathrm{Uniform}[0,1]$
and $U$ is independent of $(X,Y)$ then $(X, g(X,U))$ has the same joint distribution as $(X,Y)$.

Suppose the statement of the lemma holds for $\gamma_0, \ldots, \gamma_{n-1}$ with
jointly distributed random variables $(X_0, \ldots, X_{n-1})$. Using Lemma \ref{lem:bestcoupling},
we find a coupling $(X'_{n-1},X'_n)$ of $\gamma_{n-1}$ with $\gamma_n$ such that
$\Was{\gamma_{n-1},\gamma_n}^2 = \E{d(X'_{n-1},X'_n)^2}$. Let $U \sim \mathrm{Uniform}[0,1]$ be
independent of all the random variables $X_0, \ldots, X_{n-1}, X'_{n-1}$ and $X'_n$. Let $g$ be as
mentioned in the fact above for the pair $(X'_{n-1},X'_n)$.

Since $X_{n-1}$ has the same distribution as $X'_{n-1}$, and $U$ is independent of all other random variables,
the pair $(X_{n-1},g(X_{n-1},U))$ has the same distribution as $(X'_{n-1},X'_n)$. Let $X_n = g(X_{n-1},U)$.
Thus, $\Was{\gamma_{n-1},\gamma_n}^2 = \E{d(X'_{n-1},X'_n)^2} = \E{d(X_{n-1},X_n)^2}$. The random
variables $(X_0, \ldots, X_n)$ provide the desired coupling.
\end{proof}

Now suppose $\gamma$ is a $\M(K)$ valued path. We may choose a sequence of nested
finite partitions $\Par_0 \subset \Par_1 \ldots$ such that $\Delta(\Par_n) \to 0$ and
$\len{\gamma,\Par_n} \nearrow \len{\gamma}$.

For each $n$, we apply Lemma \ref{lem:finiterealization} to find coupled random variables
$(X_n(t); t \in \Par_n)$ such that if $\Par_n = \{ 0 = t_0 < \ldots < t_k = 1\}$ then
$$\E{d(X_n(t_i),X_n(t_{i-1}))^2} = \Was{\gamma(t_i),\gamma(t_{i-1})}^2 \;\;\text{for every}\;\; 1 \leq i \leq k.$$
Fix an $x_0 \in K$. Set $\Par_{\infty} = \cup_n \Par_n$ and extend $X_n$ to $\Par_{\infty}$ by setting
$X_n(t) \equiv x_0$ if $t \in \Par_{\infty} \setminus \Par_n$.

The process $X_n$ takes values in $K^{\Par_{\infty}}$ for every $n$. As $K^{\Par_{\infty}}$ is compact
in the product topology, by applying Prokhorov's Theorem we can find a subsequence $n_i \to \infty$
and a process $(X(t); t \in \Par_{\infty)}$ such that $X_{n_i} \to X$ weakly. As the partitions $\Par_n$
are nested we may assume w.l.o.g.~that $n_i = n$, that is, $X_n \to X$ weakly.

Consider the process $(X(t); t \in \Par_{\infty})$. We must extend $X$ continuously from the dense subset
$\Par_{\infty}$ to $[0,1]$. First, we show that $X$ has finite energy along $\Par_{\infty}$. Let
$$\len{X,\Par_{\infty}} := \lim_{n \to \infty} \len{X, \Par_n},$$
which exists by monotonicity.

\begin{lemma} \label{lem:discreteenergy}
The process $(X(t); t \in \Par_{\infty})$ satisfies $\len{X,\Par_{\infty}} \leq \len{\gamma}$.
Moreover, for every $s < t$ in $\Par_{\infty}$, $\E{d(X(t),X(s))^2} \leq (t-s) \len{\gamma, [s,t]}$.
\end{lemma}

\begin{proof}
We begin by showing $\E{d(X(t),X(s))^2} \leq (t-s) \len{\gamma, [s,t]}$ for $s < t$ in $\Par_{\infty}$.
Suppose that $s < t$ are both in $\Par_{\infty}$. From weak convergence of the $X_n$ and
compactness of $K$ we have that $\E{d(X(t),X(s))^2} = \lim_{n \to \infty} \E{d(X_n(t),X_n(s))^2}$.
We now bound $\E{d(X_n(t),X_n(s))^2}$. As $s,t \in \Par_{\infty}$, there is an $N$ such that
$s,t \in \Par_n$ for $n \geq N$. Suppose that the points of $\Par_n$ between $s$ and $t$ are
$s = t_{0,n} < t_{1,n} < \ldots < t_{k_n,n} = t$. Using Lemma \ref{lem:refinement} we deduce that for $n \geq N$,
\begin{align*}
\frac{\E{d(X_n(t),X_n(s))^2}}{t-s} &\leq \sum_{i=1}^{k_n} \frac{\E{d(X_n(t_{n,i}),X_n(t_{n,i-1}))^2}}{t_{n,i} - t_{n,i-1}} \\
& = \sum_{i=1}^{k_n} \frac{\Was{\gamma(t_{n,i}),\gamma(t_{n,i-1})}^2}{t_{n,i} - t_{n,i-1}} \\
& \leq \len{\gamma,[s,t]}.
\end{align*}
The last inequality follows due to the $t_{n,i}$ forming a partition of $[s,t]$. By letting $n \to \infty$ we
conclude from the above estimate that $\E{d(X(t),X(s))^2} \leq (t-s) \, \len{\gamma,[s,t]}$.

For the partition $\Par_n = \{ 0 = t_0 < \ldots < t_n =1\}$ we deduce from the inequality above that
\begin{equation} \label{eqn:Xenergybound}
\sum_{i=1}^n \frac{\E{d(X(t_i),X(t_{i-1}))^2}}{t_i - t_{i-1}} \leq \sum_{i=1}^n \len{\gamma,[t_{i-1},t_i]}.
\end{equation}
From the inequality (\ref{eqn:energyinequality}) we now deduce that $\sum_{i=1}^n \len{\gamma,[t_{i-1},t_i]} \leq \len{\gamma}$.
Therefore, $\len{X, \Par_{\infty}} \leq \len{\gamma}$ as required.
\end{proof}

We now show that $X$ has a continuous extension to a process defined for times $t \in [0,1]$.
Let $(\Omega, \Sigma,\mu)$ denote the probability space where $(X(t), t \in \Par_{\infty})$
is jointly defined and let $X_{\omega}(t)$ denote the outcome of $X(t)$ for $\omega \in \Omega$.
The inequality $\len{X, \Par_{\infty}} \leq \len{\gamma}$ from Lemma \ref{lem:discreteenergy}
implies that for $\mu$-almost every $\omega$ the energy of the discrete $K$-valued path
$(X_{\omega}(t), t \in \Par_{\infty})$ is finite.
%In other words, for $\mu$-almost every $\omega$,
%$$ \sup_n \, \sum_{t_i \in \Par_n} \, \frac{d^2(X_{\omega}(t_i), X_{\omega}(t_{i-1}))}{t_i-t_{i-1}} < \infty.$$
In particular, for $\mu$-almost every $\omega$ there exists a constant $C_{\omega}$ such that
$$d(X_{\omega}(t),X_{\omega}(s)) \leq C_{\omega}\sqrt{|t-s|} \;\;\text{for}\;\; s,t \in \Par_{\infty}.$$
Since $\Par_{\infty}$ is a dense subset of $[0,1]$, Lemma \ref{lem:extension} from the Appendix
implies that $(X_{\omega}(t); t \in \Par_{\infty})$ has a continuous extension to times $t \in [0,1]$ for $\mu$-almost
every $\omega$. We denote this extension by $X = (X(t), 0\leq t \leq 1)$,
which is then a $K$-valued stochastic process with continuous sample paths.

Now we show that $X$ realizes $\gamma$. Certainly, $X(t) \sim \gamma(t)$ for
$t \in \Par_{\infty}$ because $X_n(t) \to X(t)$ weakly and $X_n(t) \sim \gamma(t)$
for all large $n$ due to the partitions $\Par_n$ being nested. Suppose that $t \in [0,1] \setminus \Par_{\infty}$.
Choose a sequence $t_n \in \Par_n$ such that $t_n \to t$. By continuity of $X$ and the bounded
convergence theorem we conclude that $\E{d(X(t_n),X(t))} \to 0$. This implies that $X(t_n) \to X(t)$ weakly.
The distribution of $X(t_n)$ is $\gamma(t_n)$ and $\gamma(t_n) \to \gamma(t)$ weakly because the path
$\gamma$ is continuous due to having finite energy. Therefore, $X(t) \sim \gamma(t)$ for every $t$.

Finally we show that $\len{X} = \len{\gamma}$. As $X$ realizes $\gamma$,
$\E{d(X(t),X(s))^2} \geq \Was{\gamma(t),\gamma(s)}^2$. Hence, $\len{X} \geq \len{\gamma}$.
To get the reverse inequality first recall from Lemma \ref{lem:discreteenergy} that
$$\E{d(X(t),X(s))^2} \leq (t-s) \, \len{\gamma,[s,t]} \;\; \text{for every}\;\;s,t \in \Par_{\infty}\; \text{with}\; s < t.$$
Suppose $s < t$ are two arbitrary points in $[0,1]$. Choose sequences $\{s_n\}$ and $\{t_n\}$
such that $s_n, t_n \in \Par_n$, $s_n \leq t_n$, $s_n \searrow s$ and $t_n \nearrow t$.
From continuity of $X$ and the bounded convergence theorem we have that
$\E{d(X(t),X(s))^2} = \lim_{n \to \infty} \E{d(X(t_n),X(s_n))^2}$. Since
$$\E{d(X(t_n),X(s_n))^2} \leq (t_n-s_n)\, \len{\gamma,[s_n,t_n]}
\;\;\text{and}\;\; \len{\gamma,[s_n,t_n]} \leq \len{\gamma,[s,t]},$$
we conclude that $\E{d(X(t),X(s))^2} \leq (t-s) \, \len{\gamma,[s,t]}$ for every $s \leq t$.
For an arbitrary partition $\Par = \{ 0 = t_0 < \ldots < t_n =1\}$ we use this inequality
to deduce that $\len{X,\Par} \leq \sum_{i=1}^n \len{\gamma,[t_{i-1},t_i]}$.
The inequality (\ref{eqn:energyinequality}) implies that
$\sum_{i=1}^n \len{\gamma,[t_{i-1},t_i]} \leq \len{\gamma}$.
As $\Par$ was arbitrary it follows that $\len{X} \leq \len{\gamma}$.
This completes the proof.

%%%%%%%%%%%%%%%%%%%%%%%%%%%%%%%%%%%%%%%%%%%%%%%%%%%%%%%
%%%%%%%%%%%%%%%%%%%%%%%%%%%%%%%%%%%%%%%%%%%%%%%%%%%%%%%%
%%%%%%%%%%%%%%%%%%%%%%%%%%%%%%%%%%%%%%%%%%%%%%%%%%%%%%%%
\section{Minimal energy permuton processes from identity to reverse} \label{sec:sinecurve}

In this section we prove Theorem \ref{thm:sinecurve} and Corollary \ref{speed}.
The proof of Theorem \ref{thm:sinecurve} uses the following lemma about minimal energy paths
on a Hilbert sphere. A proof is provided in the Appendix.

\begin{lemma} \label{lem:Hilbert}
Let $\gamma$ be a path on the unit sphere of a Hilbert space between two
antipodal points $\gamma(0)$ and $-\gamma(0) = \gamma(1)$. Then
$\len{\gamma} \geq \pi^2$ with equality if and only if
$$\gamma(t) =  \cos(\pi t) \gamma(0) + \sin(\pi t) \gamma(1/2).$$
\end{lemma}

\begin{proof}[Proof of Theorem \ref{thm:sinecurve}.]
Suppose $X = (X(t); 0 \leq t \leq 1)$ is a permuton process with $X(0) = - X(1)$. Since
$\E{X(t)^2} = 1/3$, the process $X$ is a path between two antipodal points on the sphere
of radius $1/\sqrt{3}$ in the Hilbert space $L^2(\Omega, \Sigma, P)$, where
$(\Omega, \Sigma, P)$ is the probability space over which the process $X$ is defined.
From Lemma \ref{lem:Hilbert} we see that $\len{X} \geq \pi^2/3$ with equality if and
only only if $X(t) = \cos(\pi t) X(0) + \sin(\pi t) X(1/2)$. In case of equality, since
$X(t) \sim \Uni$ for every $t$, this equation for $X$ implies that the projection of $(X(0),X(1/2))$
onto any line through the origin has the $\Uni$ distribution. Thus, $(X(0),X(1/2))$ is distributed
according to the Archimedean measure and $X$ is the Archimedean process.
\end{proof}

\begin{proof}[Proof of Corollary \ref{speed}.]
Theorem \ref{thm:sinecurve} implies that all limit points of $X^n$ are supported on permuton processes
with energy at least $\pi^2/3$. Due to the uniqueness of energy minimizers, $X^n$
will converge to the Archimedean process if the expected energy of any limit point $X$ of $X^n$
is at most $\pi^2/3$.
By bounded convergence theorem and assumption \eqref{eqn:RSNconjecture} we have
$$
\E {(X(t)-X(0))^2} \le \frac{(\pi t)^2}{3}.
$$
The $\eps_n$-stationarity of $X^n$ implies $\eps$-stationarity of $X$ for all $\eps>0$, so
$$
\E {(X(t)-X(s))^2} \le \frac{(\pi (s-t))^2}{3},
$$
and therefore for every partition $\Par = \{t_0 = 0, \ldots, t_k = 1\}$ we have
\[
\sum_{i=1}^k \frac{\E{(X(t_i)-X(t_{i-1}))^2}}{t_i - t_{i-1}} \leq \frac{\pi^2}{3}. \qedhere
\]
\end{proof}

Note that conversely, the Archimedean process limit and the bounded convergence theorem would
imply that for every $t$, $\E{(X^n(t)-X^n(0))^2} \to \frac{2}{3}(1-\cos(\pi t))$.

%%%%%%%%%%%%%%%%%%%%%%%%%%%%%%%%%%%%%%%%%%%%%%%%%%%%%%%
%%%%%%%%%%%%%%%%%%%%%%%%%%%%%%%%%%%%%%%%%%%%%%%%%%%%%%%%
%%%%%%%%%%%%%%%%%%%%%%%%%%%%%%%%%%%%%%%%%%%%%%%%%%%%%%%%
\section{Minimal energy permuton paths from identity to reverse} \label{sec:ugproof}

In this section, we prove Theorem \ref{thm:uniquegeodesic}.
Using Theorem \ref{thm:realization} we transfer the study of paths in
$\Pspace$ to $[-1,1]^2$--valued stochastic processes.
Then we solve the corresponding energy minimization problem for stochastic processes.
We verify that there is an unique energy minimizer and its fixed time marginals
agree with the Archimedean path.

\paragraph{\textbf{1-dimensional energy minimization}}
In proving Theorem \ref{thm:uniquegeodesic} we will reduce the 2-dimensional
energy minimization problem to a pair of 1-dimensional energy minimization problems.
Here we solve that 1-dimensional problem.

For continuous $f : [0,1] \to \R$, let $\len{f}$ and $\len{f, \Par}$ denote energy
w.r.t.~the Euclidean metric on $\R$.

\begin{lemma} \label{lem:1denergy}
	Let $X = (X(t); 0 \leq t \leq 1)$ be a continuous $\R$-valued stochastic process such that $\E{X(t)} = 0$ and
	$\E{X(t)^2} < \infty$ for every $t$. Set $\sig(t) = \E{X(t)^2}^{1/2}$. Then $\len{X} \geq \len{\sig}$.
	Here the energy of $X$ is w.r.t.~the $L^2$-metric and the energy of $\sig$ is w.r.t.~the Euclidean metric.
	
	Moreover, suppose that $\len{X} = \len{\sig}$, $\len{\sig} < \infty$ and $\sig(t) > 0$ for $t > 0$.
	Then the following holds almost surely,
	$$X(t) = \frac{\sig(t)}{\sig(1)}\,X(1)\;\;\text{for every}\;\; 0 \leq t \leq 1.$$
\end{lemma}

\begin{proof}
	We have that $\E{|X(t)-X(s)|^2} = \sig(t)^2 - 2\E{X(t)X(s)} + \sig(s)^2$. The Cauchy-Schwarz inequality implies
	$\E{X(t)X(s)} \leq \sig(t) \sig(s)$, and hence, $\E{|X(t)-X(s)|^2} \geq (\sig(t)-\sig(s))^2$.
	From this inequality it is immediate that $\len{X} \geq \len{\sig}$.
	
	Now suppose that $\len{X} = \len{\sig}$, $\len{\sig}$ is finite and $\sig(t) > 0$ for every $t > 0$.
	If we show that $X(t)/\sig(t)$ is almost surely constant on the interval $[\eps, 1]$, for any $ \eps > 0$,
	then the continuity of $X$ implies that $X(t)/\sig(t)$ is almost surely constant on $[0,1]$.
	Therefore, we may assume that $\sig(t) > 0$ for $t \in [0,1]$.
	
	Set $\delta = \inf_{t \in [0,1]}\, \sig(t)$. Then $\delta > 0$ since $\sig(t)$ is continuous and positive on $[0,1]$.
	%As $v$ is continuously differentiable, $\len{\sig, [a,b]\cap \Par} = \int_a^b \sig'(t)^2 \,dt$.
	Set $Y(t) = \frac{X(t)}{\sig(t)}$. For $0 \leq s \leq t \leq 1$,
	\begin{align*}
		\E{|Y(t)-Y(s)|^2} &= \frac{\E{|X(t)-X(s)|^2} - |\sig(t)-\sig(s)|^2}{\sig(s)\sig(t)} \\
		& \leq \frac{\E{|X(t)-X(s)|^2} - |\sig(t)-\sig(s)|^2}{\delta^2}.
	\end{align*}
	The estimate above implies that for any finite partition $\Par$ of $[0,1]$,
	$$\len{Y,\Par} \leq \delta^{-2} \big(\len{X, \Par} - \len{\sig,\Par} \big) \leq \delta^{-2} \big (\len{X} - \len{\sig,\Par} \big).$$
	Choose a sequence of nested partitions $\Par_0 \subset \Par_1 \subset \cdots$ such that
	$\len{\sig, \Par_n} \to \len{\sig}$.
	We deduce from the above that $\len{Y,\Par_n} \to 0$ due to $\len{X} = \len{\sig}$.
	Since $\len{Y,\Par_n}$ is monotone increasing we conclude that $\len{Y,\Par_n} = 0$ for every $n$.
	Set $\Par = \cup_n \Par_n$. Then for every $s,t \in \Par$,
	$$\E{|Y(t)-Y(s)|^2} \leq |t-s| \cdot \left (\sup_n \,\len{Y,\Par_n} \right) = 0.$$
	
	Fix an arbitrary $p \in \Par$. We deduce from the above that for every $q \in \Par$, $\pr{Y(q) = Y(p)} = 1$.
	Taking the countable intersection of these events over all $q \in \Par$ we conclude that
	$$\pr{Y(q) = Y(p) \;\text{for every}\; q \in \Par} = 1.$$
	The continuity of $Y$ and the fact that $\Par$ is a dense subset of $[0,1]$ imply that almost surely,
	$Y(t) \equiv Y(p)$ for every $t \in [0,1]$.
	In other words, $X(t) = \frac{\sig(t)}{\sig(1)}\, X(1)$ for every $t$, almost surely, as required.
\end{proof}

\paragraph{\textbf{Proof of Theorem \ref{thm:uniquegeodesic}}}

Let $\perm = (\perm(t); 0 \leq t \leq 1)$ be a path in $\Pspace$ from $\id$ to $\rev$ such that
$\len{\perm} < \infty$ in the Wasserstein metric. Using Theorem \ref{thm:realization} we may realize
$\perm$ as a $[-1,1]^2$-valued continuous stochastic process $X$ such that $\len{\perm} = \len{X}$.

Write $X(t) = (x(t),y(t))$. Then $x(t)$ and $y(t)$ are distributed as $\mathrm{Uniform}[-1,1]$
since $\perm(t)$ is a permuton. Also, $x(0) = y(0)$ and $x(1) = -y(1)$ due to $\perm(0) = \id$
and $\perm(1) = \rev$. Set
\begin{align*}
	u(t) &= \frac{x(t) - y(t)}{\sqrt{2}},\\
	v(t) & = \frac{x(t) + y(t)}{\sqrt{2}}.
\end{align*}

Then $\E{u(t)} = \E{v(t)} = 0$ for every $t$. For the boundary conditions we have
$u(0) = 0$ and $u(1) = \sqrt{2}x(1)$, while $v(0) = \sqrt{2}x(0)$ and $v(1) = 0$.
Set $$\sig^2_u(t) = \E{u(t)^2} \;\;\text{and}\;\; \sig^2_v(t) = \E{v(t)^2}.$$
Since  $u(t)^2 + v(t)^2 = x(t)^2 + y(t)^2$, we see that $\sig^2_u(t)+ \sig^2_v(t) = \E{x(t)^2 + y(t)^2} = 2/3$
due to $x(t)$ and $y(t)$ being distributed as $\mathrm{Uniform}[-1,1]$.

The map $t \to (\sig_u(t), \sig_v(t))$ is a path on the circle of radius $\sqrt{2/3}$
that begins at $(0,\sqrt{2/3})$ and ends at $(\sqrt{2/3},0)$.
It is well known that there is a unique path of minimal energy on such a circle
from $(0,\sqrt{2/3})$ to $(\sqrt{2/3},0)$. This is the
minor arc going from $(0,\sqrt{2/3})$ to $(\sqrt{2/3},0)$, and uniquely
parametrized by $t \to \sqrt{2/3}\,(\sin(\frac{\pi}{2}t), \cos(\frac{\pi}{2}t))$.
The energy of this path is
$$\frac{\pi^2}{6}  \int_0^1 \cos \left (\frac{\pi}{2}t \right )^2 +\sin \left (\frac{\pi}{2}t \right )^2 \;dt =  \frac{\pi^2}{6}.$$
Consequently, $\len{(\sig_u,\sig_v),\Par}^2 \geq \pi^2/6$. In case of equality
we must have
\begin{equation} \label{eqn:equalitysd}
\sig_u(t) = \sqrt{2/3} \, \sin \big(\frac{\pi}{2}t \big) \;\;\text{and}\;\;
\sig_v(t) = \sqrt{2/3} \, \cos \big(\frac{\pi}{2}t \big).
\end{equation}

Note that $\E{|X(t)-X(s)|^2} = |u(t)-u(s)|^2 + |v(t)-v(s)|^2$, where the distance for
$X$ is in the Euclidean metric of $\R^2$. This implies that $\len{X} = \len{u} + \len{v}$.
Lemma \ref{lem:1denergy} then implies $\len{u} + \len{v} \geq \len{\sig_u} + \len{\sig_v}$.
Therefore,
$$\len{\gamma} = \len{X} = \len{u} + \len{v} \geq \len{\sig_u} + \len{\sig_v}.$$
However, $\len{\sig_u} + \len{\sig_v} = \len{(\sig_u,\sig_v)} \geq \pi^2/6$.

We have deduced that $\len{\gamma} \geq \pi^2/6$ for any path $\gamma$ in $\Pspace$.
As $\len{\Abf} = \pi^2/6$, we deduce that $(\Abf(t); 0 \leq t \leq 1)$ has minimal energy
among all paths from $\id$ to $\rev$ in $\Pspace$. If $\len{\gamma} = \pi^2/6$ then the
functions $\sig_u$ and $\sig_v$ must equal the functions from (\ref{eqn:equalitysd}).
In this case we may apply the case of equality from Lemma \ref{lem:1denergy}
to the processes $u(t)$ and $v(t)$ to conclude that $u(t) = \sqrt{2} \sin(\frac{\pi}{2}t) \, x(1)$
and $v(t) = \sqrt{2} \cos(\frac{\pi}{2}t) \, x(0)$ for every $t$, almost surely. In terms of $X$
we get that almost surely, for all $0 \leq t \leq 1$,
\begin{align}
	x(t) &= \cos \left (\frac{\pi}{2}t \right ) x(0) + \sin \left (\frac{\pi}{2}t \right ) x(1) \label{eqn:arch} \\
	y(t) &= \cos \left (\frac{\pi}{2}t \right ) x(0) - \sin \left (\frac{\pi}{2}t \right ) x(1) \nonumber.
\end{align}

We claim that (\ref{eqn:arch}) implies $(x(0),x(1))$ is distributed as the Archimedean measure.
If this holds then we have $\gamma(t) = \Abf(t)$ because $X(t) \sim \gamma(t)$ and
(\ref{eqn:arch}) implies that $X(t) \sim \Abf(t)$.  The latter holds because the formula
above implies the density function of $X(t)$ agrees with that of $\Abf(t)$ if $(x(0), x(1))$
is distributed as the Archimedean measure.

To see that $(x(0),x(1))$ is distributed as the Archimedean measure
observe that $x(t)$ is the projection of $(x(0),x(1))$ onto the line through the origin
with angle $\frac{\pi}{2}t$. Also, $y(t)$ is the projection of $(x(0),x(1))$ on the line
through the origin with angle $-\frac{\pi}{2}t$. As $x(t)$ and $y(t)$ are distributed
as $\Uni$ for every $t$, it follows that the distribution of the projection
of $(x(0),x(1))$ onto any line through the origin is $\Uni$.
This property determines the Archimedean measure.
This completes the proof of Theorem \ref{thm:uniquegeodesic}.

%%%%%%%%%%%%%%%%%%%%%%%%%%%%%%%%%%%%%%%%%%%%%%%
%%%%%%%%%%%%%%%%%%%%%%%%%%%%%%%%%%%%%%%%%%%%%%%
%%%%%%%%%%%%%%%%%%%%%%%%%%%%%%%%%%%%%%%%%%%%%%%
\section{Permutons in Wasserstein metric} \label{sec:permutongeometry}

In the final section of the paper we establish a formula for the Wasserstein distance
from the identity to any permuton and use it to prove Theorem \ref{thm:distancesquared}.

\begin{lemma} \label{lem:permW2}
Let $\sigma$ and $\tau$ be permutations in $\Sn$. Then,
\[\Was{\mu_{\sigma}, \mu_{\tau}} ^2= \frac{4}{n^3}\, \inf_{\pi \in \Sn} \sum_i (i - \pi(i))^2 + (\sigma(i) -\tau(\pi(i)))^2\,.\]
\end{lemma}
The proof is in the Appendix.

\begin{theorem} \label{lem:identitydistance}
The Wasserstein distance of the identity permuton $\id$ from any permuton $\perm = (X,Y)$ is as follows.
Let $(X',Y')$ denote an independent copy of $(X,Y)$. Then,
\[\Was{\id,\perm}^2 = \frac{4}{3} - 2\E{\max \, \{ X+Y, X'+Y' \}}\,.\]
\end{theorem}

\begin{proof}
We first derive the analogue of the above formula for permutations and then take limits
to get the final result. For a permutation $\sigma \in \Sn$, consider
$$ \sum_i (i - \pi(i))^2 + (i -\sigma \pi(i))^2 =  4 \sum_i i^2 -2\sum_i i \cdot (\pi(i) + \sigma \pi(i)).$$
Reindexing the latter sum by setting $i := \pi^{-1}(i)$ and replacing $\pi$ by $\pi^{-1}$ we get that
\[\inf_{\pi \in \Sn} \sum_i (i - \pi(i))^2 + (i -\sigma \pi(i))^2 = 4 \sum_i i^2 - 2\sup_{\pi} \sum_i \pi(i)(i + \sigma(i))\,.\]

The sum $\sum_i \pi(i) \cdot (i + \sigma(i))$ is maximized by choosing $\pi(i)$ to be the rank of $i+\sigma(i)$
in the sequence $1+\sigma(1), \ldots, n + \sigma(n)$.
We can write the maximizing permutation $\pi$ as $\pi(i) = \sum_j \ind{j + \sigma(j) \leq i + \sigma(i)}$, whence,
\begin{align*} \sup_{\pi} \sum_i \pi(i) \cdot (i + \sigma(i)) &= \sum_i \sum_j \ind{j + \sigma(j) \leq i + \sigma(i)} (i + \sigma(i))\\
&= \frac{1}{2} \sum_{i,j} \max \, \{i + \sigma(i), j + \sigma(j)\}\,. \end{align*}
Therefore,
\begin{equation} \label{eqn:wassidentity}
\inf_{\pi \in \Sn} \sum_i (i - \pi(i))^2 + (i -\sigma \pi(i))^2 = 4 \sum_i i^2 - \sum_{i,j} \max \, \{i + \sigma(i), j + \sigma(j)\}.
\end{equation}

Let $\mu_{\sig}$ be the empirical distribution associated to $\sigma$. If $(X,Y)$ and $(X',Y')$ are
two independent random variables with distribution $\mu_{\sig}$ then from (\ref{eqn:wassidentity}) and elementary simplifications we get
\begin{align*} \E{\max \, \{X+Y,X'+Y'\}} &= \frac{2}{n^2} \sum_{i,j} \max \left \{\frac{i + \sigma(i)}{n}-1, \frac{j + \sigma(j)}{n}-1\right \}\\
&= \frac{2}{n^3} \big [4 \sum_i i^2 - \inf_{\pi} \sum_i (i - \pi(i))^2 + (i -\sigma(\pi(i)))^2 \big ] - 2\\
&= \frac{2}{n^3} \big [4 \sum_i i^2 - \frac{n^3}{4}\,\Was{\mu^{\mathrm{id}^n},\mu^{\sigma}}^2 \big ] - 2,
\end{align*}
where the last equality follows from Lemma \ref{lem:permW2}. Since $\sum_i i^2 = n^3/3 + O(n^2)$,
the above implies
$$\Was{\mu^{\mathrm{id}^n},\mu^{\sigma}}^2 = \frac{4}{3} - 2 \E{\max \, \{X+Y,X'+Y'\}} + O(1/n).$$

There exits permutations $\sigma^n \in \Sn$ such that $\mu^{\sig^n}$ converges to $\perm$
in the Wasserstein distance by \cite[Theorem 1.6]{permuton} and Lemma \ref{lem:Wconvergence}.
Therefore, $\Was{\id,\perm}^2 = \lim_n \E{\Was{\mu^{\mathrm{id}^n},\mu^{\sig^n}}^2}$.
The formula for $\Was{\id,\perm}^2$ follows from this convergence upon taking the large $n$ limit
of the formula above.
\end{proof}

Observe that for a permuton $\perm = (X,Y)$ we have $\Was{(X,Y),\rev} = \Was{(X,-Y),\id}$.
From Theorem \ref{lem:identitydistance} we conclude that for any permuton $\perm = (X,Y)$,
\begin{equation} \label{eqn:sumdistancesquared}
 \Was{\id,\perm}^2 + \Was{\perm,\rev}^2 = \frac{8}{3} - 2\,\E{\max \, \{X+Y, X'+Y'\} + \max \, \{X-Y, X'-Y'\}}\,.
\end{equation}

\begin{proof}[Proof of Theorem \ref{thm:distancesquared}.]
Given $\perm = (X,Y)$, let $W = \frac{X-Y}{\sqrt{2}}$, $V = \frac{X+Y}{\sqrt{2}}$. Define $(W',V')$
analogously for the pair $(X',Y')$. From (\ref{eqn:sumdistancesquared}) we have
$$\Was{\id,\perm}^2 + \Was{\perm,\rev}^2 = \frac{8}{3} - 2\sqrt{2}\,\E{\max \, \{W,W'\} + \max \, \{V,V'\}}.$$

Suppose $Z$ is an integrable random variable and $Z'$ is an independent copy of $Z$. Then
$$\E{\max \{Z,Z'\}} = 2\E{Z \, \mathbb{P}_{Z'}[Z' < Z]} + \E{Z \, \mathbb{P}_{Z'}[Z'=Z]}.$$
Define $F(z,u) = \pr{Z < z} + u\pr{Z=z}$ for $z \in \R$ and $0 < u < 1$. The function $F$
is called the ``distributional transform" of $Z$. If $U$ is independent of
$Z$ and distributed as $\uni{0,1}$ then $F(Z,U)$ is distributed as $\uni{0,1}$ as well \cite[Proposition 2.1]{R}.
From the definition of $F(z,u)$ we obtain
$$\E{Z \,\pr{Z' < Z}} =\E{Z F(Z,U)} - \frac{1}{2} \E{Z \,\mathbb{P}_{Z'}[Z=Z']}.$$
In particular, $\E{\max\{Z,Z'\}} = 2 \E{Z F(Z,U)}$. Hence,
$$\E{\max \, \{W,W'\} + \max \, \{V,V'\}} = 2 \E{WF(W,U) + VG(V,U)},$$
where $F$ and $G$ are the distributional transforms of $W$ and $V$.

Set $\hat{F}(x,u) = 2F(x,u) -1$ and $\hat{G}(x,u) = 2G(x,u) -1$. Observe that
$$WF(W,U) +VG(V,U) = \frac{W\hat{F}(W,U)}{2} + \frac{W}{2} + \frac{V\hat{G,U}(V)}{2} + \frac{V}{2}.$$
We take expectations of this equation and use that $\E{W} = \E{V} = 0$. Then in order to bound
the expectation of the r.h.s.~we use the inequality $ab \leq \frac{a^2 + b^2}{2}$.
This gives that $W\hat{F}(W,U) \leq (W^2 + \hat{F}(W,U)^2)/2$ and $V\hat{G}(V,U) \leq (V^2 + \hat{G}(V,U)^2)/2$.
Since $\E{W^2 + V^2} = \E{X^2 + Y^2} = 2/3$, we conclude that
$$\E{WF(W,U) +VG(V,U)} \leq \frac{1}{4} \, \E{W^2 + \hat{F}(W,U)^2 + V^2 + \hat{G}(V,U)^2} = \frac{1}{3}.$$
Furthermore, there is equality if and only if $W = \hat{F}(W,U)$ and $V = \hat{G}(V,U)$,
which is equivalent to $(W,V)$ being a permuton. As a result,
$$\Was{\id,\perm}^2 + \Was{\perm,\rev}^2 \geq \frac{8- 4\sqrt{2}}{3},$$
with equality if and only if $\big ( \frac{X-Y}{\sqrt{2}}, \frac{X+Y}{\sqrt{2}} \big )$ is a permuton.
\end{proof}

\section*{Acknowledgements}
M.~Rahman was partially supported by an NSERC PDF award.
B.~Vir\'{a}g was supported by the Canada Research Chair program,
the NSERC Discovery Accelerator grant, the MTA Momentum Random Spectra research group,
and the ERC consolidator grant 648017 (Ab\'{e}rt). M.~Vizer was supported by NKFIH
under the grant SNN 116095.

\section{Appendix} \label{sec:appendix}

\paragraph{\textbf{Proof of Lemma \ref{lem:processfacts}}}

\begin{proof}
Note that $m^{\delta}(Y)$ is non increasing in $\delta$ and converges to zero as $\delta \to 0$ almost surely
since $t \to Y(t)$ is uniformly continuous. Also, $m^{\delta}(Y) \leq 2$, and thus, $\E{m^{\delta}(Y)} \to 0$
as $\delta \to 0$ by the bounded convergence Theorem. For the second claim observe that
$\big | |Y(t)-Y(s)| - |\hat{Y}(t)-\hat{Y}(s)| \big | \leq 2||Y-\hat{Y}||_{\rm{\infty}}$ by the triangle inequality.
Therefore, $m^{\delta}(Y) \leq m^{\delta}(\hat{Y}) +  2||Y-\hat{Y}||_{\rm{\infty}}$ and
vice-versa, which implies the claim. Finally, for the third claim notice that
$$| \mathrm{Lin}(n,Y)(t) - Y(t)| = \sum_{i=1}^{n}
| Y(\frac{i}{n})-Y(t) + n(Y(\frac{i-1}{n})-Y(\frac{i}{n}))(t-\frac{i}{n})| \mathbf{1}_{[\frac{i-1}{n},\frac{i}{n}]}(t) \leq 2m^{1/n}(Y).$$
The claim now follows from the assertion of the first claim.
%$$\rm{Lin}(n,Y)(t) = \sum_{i=1}^{n} n\,[Y(\frac{i-1}{n})-Y(\frac{i}{n})] (t-\frac{i}{n})\ind{[\frac{i-1}{n},\frac{i}{n}]}(t).$$
\end{proof}

\paragraph{\textbf{Proof of Lemma \ref{lem:refinement}}}
\begin{proof}
Suppose that $\Par = \{ 0 = t_0 < t_1 \ldots < t_n = 1\}$.
As $\Par'$ is a refinement of $\Par$ it contains points between the $t_i$. Suppose the points of
$\Par'$ are indexed as $t_{i,j}$ for $0 \leq i \leq n$ and $0 \leq j \leq k_i$ such
that $t_i = t_{i,0} < t_{i,1} < \ldots < t_{i,k_i} = t_{i+1,0} = t_{i+1}$. From the triangle inequality,
$d(\gamma(t_i),\gamma(t_{i-1})) \leq \sum_{j=1}^{k_{i-1}} d(\gamma(t_{i-1,j}),\gamma(t_{i-1,j-1}))$.
So we deduce from the Cauchy-Schwarz inequality that
\begin{align*}
d(\gamma(t_i),\gamma(t_{i-1})) &\leq \sum_{j=1}^{k_{i-1}} \sqrt{t_{i-1,j} - t_{i-1,j-1}}\quad
\frac{d(\gamma(t_{i-1,j}),\gamma(t_{i-1,j-1}))}{\sqrt{t_{i-1,j} - t_{i-1,j-1}}}\\
%& \leq \left [\sum_{j=1}^{k_{i-1}} (t_{i-1,j} - t_{i-1,j-1}) \right ]^{1/2}
%\left [ \sum_{j=1}^{k_{i-1}} \frac{d(\gamma(t_{i-1,j}),\gamma(t_{i-1,j-1}))^2}{t_{i-1,j} - t_{i-1,j-1}} \right ]^{1/2}\\
&\leq \sqrt{t_i - t_{i-1}}  \; \left [ \sum_{j=1}^{k_{i-1}} \frac{d(\gamma(t_{i-1,j}),\gamma(t_{i-1,j-1}))^2}{t_{i-1,j} - t_{i-1,j-1}} \right ]^{1/2}.
\end{align*}

We conclude that
$$\frac{d(\gamma(t_i),\gamma(t_{i-1}))^2}{t_i - t_{i-1}} \leq
\sum_{j=1}^{k_{i-1}} \frac{d(\gamma(t_{i-1,j}),\gamma(t_{i-1,j-1}))^2}{t_{i-1,j} - t_{i-1,j-1}}.$$
Summing the inequality above over $i$ implies that $\len{\gamma,\Par} \leq \len{\gamma, \Par'}$.
\end{proof}

\paragraph{\textbf{Proof of Lemma \ref{lem:Hilbert}}}
\begin{proof}
Suppose that $\gamma$ is a path on the unit sphere of a Hilbert space
$H$ from the vector $\gamma(0)$ to its antipode $-\gamma(0)$. For unit vectors $a$ and $b$ it is easily seen
that $||a-b||^2 = 4\sin(\theta/2)^2$ where $\theta = \arccos(\langle a,b \rangle)$ is the unique angle
between $a$ and $b$ in the interval $[0,\pi]$. Let $\theta(s,t)$ denote the angle between
$\gamma(s)$ and $\gamma(t)$. Using the inequality $x-(x^3/6) \leq \sin(x) \leq x$ for $0 \leq x \leq \pi$,
we see that for any partition $\Par = \{0 = t_0 < \ldots < t_n = 1\}$,
$$(1-\delta(\Par))^2 \, \sum_{i=1}^n \frac{\theta(t_{i-1},t_i)^2}{t_{i-1}-t_i} \leq
 \len{\gamma,\Par} \leq \sum_{i=1}^n \frac{\theta(t_{i-1},t_i)^2}{t_{i-1}-t_i},$$
where $\delta(\Par) = \max_i \{\theta(t_{i-1},t_i)^2\}/24$. Note that $\theta(s,t)$ is continuous as
$\gamma$ is continuous, and in particular, $\delta(\Par) \to 0$ as the mesh size $\Delta(\Par) \to 0$.
This implies that $\len{\gamma} = \sup_{\Par} \sum_i \frac{\theta(t_{i-1},t_i)^2}{t_{i-1}-t_i}$.

If $a,b$ and $c$ are unit vectors in $H$ then the corresponding angles between them satisfy
$\theta(a,b) + \theta(b,c) \geq \theta(a,c)$. This elementary fact can be deduced by considering unit
vectors in $\R^3$ since $a,b$ and $c$ lie is a 3-dimensional subspace.
In particular, $$\sum_i \theta(t_{i-1},t_i) \geq \theta(0,1) = \pi$$
because $\gamma(0)$ and $\gamma(1)$ are antipodal. From the Cauchy-Schwarz inequality we conclude
that $$\pi^2 \leq \sum_i (\sqrt{t_{i-1}-t_i})^2 \cdot \sum_i \frac{\theta(t_{i-1},t_i)^2}{(\sqrt{t_{i-1}-t_i})^2} =
\sum_i \frac{\theta(t_{i-1},t_i)^2}{t_{i-1}-t_i},$$
with equality only if $\theta(t_{i-1},t_i) = \pi(t_{i-1}-t_i)$. We may now conclude that $\len{\gamma} \geq \pi^2$,
and moreover, if there is equality then $\theta(s,t) = \pi |s-t|$.

Now suppose that $\len{\gamma} = \pi^2$. We show that
$\gamma(t) =  \cos(\pi t) \gamma(0) + \sin(\pi t) \gamma(1/2)$.
From the fact that $\theta(s,t) = \pi |t-s|$ we see that $\theta(0,t) = \pi t$ and we may write
$$\gamma(t) = \cos(\pi t) \, \gamma(0) + \sin(\pi t) \, x(t),$$ where $x(t)$ is a unit vector that is
orthogonal to $\gamma(0)$. Noting that $x(1/2) = \gamma(1/2)$ we have that
$\langle \gamma(t),\gamma(1/2) \rangle = \sin(\pi t) \,\langle x(t),x(1/2) \rangle$.
Now $\langle \gamma(t),\gamma(1/2) \rangle \,= \cos(\theta(t,1/2))$, which equals
$\sin(\pi t)$ because $\theta(t,1/2) = \pi |t-\frac{1}{2}|$. Therefore, $\langle x(t),x(1/2) \rangle \,= 1$
for $0 < t < 1$, which implies that $x(t) = x(1/2)$ as both of these are unit vectors.
Consequently, $\gamma(t) = \cos(\pi t) \gamma(0) + \sin(\pi t) \gamma(1/2)$.
\end{proof}

\paragraph{\textbf{Proof of Lemma \ref{lem:permW2}}}

\begin{proof}
Couplings between $\mu_{\sigma}$ and $\mu_{\tau}$ are supported on the points
$ \left (\frac{2i}{n} -1, \frac{2\sigma(i)}{n} -1, \frac{2j}{n} -1, \frac{2 \tau(j)}{n} -1 \right )$ for $1 \leq i,j \leq n$.
Thus, a coupling $(V,W)$ between $\mu_{\sigma}$ and $\mu_{\tau}$ is always
described by the array of numbers $[\alpha_{i,j}]_{1 \leq i,j \leq n}$ such that
\begin{equation} \label{eqn:permW2a}
 \alpha_{i,j} = \pr{V = \left (\frac{2i}{n} -1, \frac{2\sigma(i)}{n} -1 \right ), W = \left (\frac{2j}{n} -1, \frac{2\tau(j)}{n} -1 \right )}.
 \end{equation}

The constraints $V \sim \mu_{\sigma}$ and $W \sim \mu_{\tau}$ is equivalent to
the matrix $M = [n \alpha_{i,j}]$ being doubly stochastic. Denoting $M = [m_{i,j}]$, we get that
\begin{equation} \label{eqn:permW2b}
\E{||V-W||^2} = \frac{4}{n^3} \sum_{i,j} m_{i,j} [(i-j)^2 + (\sigma(i)-\tau(j))^2].
\end{equation}
Let $\mathcal{B}_n$ be the set of all $n \times n$ doubly stochastic matrices.
The map taking $M = [m_{ij}] \in \mathcal{B}_n$ to the r.h.s.~of \eqref{eqn:permW2b}
is linear, and hence minimized at one of the extreme points of the convex set $\mathcal{B}_n$.
These are the permutation matrices $P_{\pi}$ for $\pi \in \Sn$. For a permutation matrix $P_{\pi}$,
we have that
$\sum_{i,j} m_{i,j} [(i-j)^2 + (\sigma(i)-\tau(j))^2] = \sum_{i} (i-\pi(i))^2 + (\sigma(i)-\tau(\pi(i)))^2$.
As a result, we conclude from (\ref{eqn:permW2b}) that
$$\Was{\mu_{\sigma}, \mu_{\tau}}^2 = \frac{4}{n^3} \left [ \inf_{\pi \in \Sn} \, \sum_{i} (i-\pi(i))^2 + (\sigma(i)-\tau(\pi(i)))^2 \right ].$$
\end{proof}

The proofs of the following lemmas use standard arguments. We omit them for brevity.

\begin{lemma} \label{lem:Wconvergence}
Let $(K,d)$ be a compact metric space.
Let $\nu_n$ be a sequence of Borel probability measures on $K$.
Then $\nu_n$ converges weakly to a measure $\nu$ if and only if $\Was{\nu_n,\nu} \to 0$.
\end{lemma}

%\begin{proof}
%Suppose that $\nu_n \to \nu$ weakly. By Skorokhod's representation theorem,
%there exists random variables $V'_n$ and $V'$ defined on a common probability space
%such that $V'_n \sim \nu_n$, $V' \sim \nu$, and $V'_n \to V$ pointwise almost surely.
%But $\Was{\nu_n,\nu}^2 \leq \E{d(V'_n,V')^2}$. Since $d(V'_n,V') \to 0$ almost surely and
%$d(V'_n,V') \leq \mathrm{diam}(\hat{K}) < \infty$, the bounded convergence Theorem implies
%that $\Was{\nu_n,\nu}^2 \to 0$.
%
%Conversely, suppose that $\Was{\nu_n,\nu} \to 0$. There exists a coupling $(V_n,V)$ of $\nu_n$ and $\nu$
%such that $\E{d(V'_n,V')^2} \leq 2 \Was{\nu_n,\nu}^2$ for every $n$. Then $\pr{d(V'_n,V') > \delta} \leq 2\Was{\nu_n,\nu}^2/\delta^2$
%by Markov's inequality. Let $f: K \to \R$ be continuous. Given $\eps > 0$ there exists a $\delta$
%such that $|f(u)-f(v)| < \eps$ if $d(u,v) < \delta$ due to uniform continuity. As such,
%$|\int f \nu_n(dx) - \int f \nu(dx)| = |\E{f(V_n)-f(V)}| \leq \eps + 2||f||_{\infty}\pr{d(V'_n,V') > \delta}
%\leq \eps + 4\frac{||f||_{\infty}\Was{\nu_n,\nu}^2}{\delta^2}$. As $\eps$ was arbitrary it follows that
%$\lim_{n \to \infty} |\int f \nu_n(dx) - \int f \nu(dx)| = 0$ for every continuos $f$, as required.
%\end{proof}

\begin{lemma} \label{lem:bestcoupling}
Let $\nu, \nu'$ be Borel probability measures on a compact metric space $(K,d)$.
There exists a coupling $(V,W)$ of $\nu$ with $\nu'$ such that $\Was{\nu,\nu'} = \E{d(V,W)^2}^{1/2}$.
\end{lemma}

%\begin{proof}
%Consider couplings $(V_n,W_n) \in K^2$ such that
%$V_n \sim \nu$, $W_n \sim \nu'$ and $\E{d(V_n,W_n)^2} \to \Was{\nu,\nu'}^2$.
%$K^2$ is a compact metric space in the product topology and so $\M(K^2)$ is
%compact the weak topology (Prokhorov's Theorem). Hence, we can find a subsequence
%$(V_{n_i},W_{n_i})$ that converges weakly to some $(V,W)$.
%As the distance function is continuous and $K$ is compact we conclude that
%$\E{d(V,W)^2} = \lim_{i \to \infty} \E{d(V_{n_i},W_{n_i})^2} = \Was{\nu,\nu'}^2$.
%\end{proof}

\begin{lemma} \label{lem:extension}
Let $K$ be a complete metric space and $S \subset [0,1]$ a countable dense set.
Suppose $f : S \to K$ has modulus of continuity $m$ on $S$, i.e., $d(f(t),f(s)) \leq m(|t-s|)$
for $s,t \in S$. Then $f$ has an extension to $[0,1]$ with modulus of continuity $m$.
\end{lemma}

%\begin{proof}
%For $t \in [0,1]$ let $t_n$ be any sequence in $S$ converging to $t$. Then $f(t_n)$
%is a Cauchy sequence in $K$ since $d(f(t_n),f(t_m)) \leq m(|t_n-t_m|) \to 0$ as $n,m \to \infty$.
%Let $f(t) \in K$ denote the limit, and note that this does not depend on the approximating
%sequence $t_n$ due to $f$ having modulus of continuity $m$ on $S$. In this manner
%$f$ extends to $[0,1]$. Also, if $s,t \in [0,1]$ and $s_n \to s, t_n \to t$ then
%$d(f(t),f(s)) = \lim_n d(f(t_n),f(s_n)) \leq \lim_n m(|t_n-s_n|)= m(|t-s|)$.
%\end{proof}

%%%%%%%%%%%%%%%%%%%%%%%%%%%%%%%%%%%%%%%%%%%%%%%%%%%%%%%
%%%%%%%%%%%%%%%%%%%%%%%%%%%%%%%%%%%%%%%%%%%%%%%%%%%%%%%
\bibliographystyle{habbrv}

\end{document}